\documentclass[11pt]{article}

\usepackage[a4paper,margin=1in]{geometry}
\usepackage{amsmath,amssymb,amsthm,mathtools}
\usepackage[bookmarks=false]{hyperref}
\hypersetup{hidelinks}
\usepackage{enumitem}
\usepackage{booktabs}
\usepackage{verbatim}
\usepackage{color}
\usepackage{parskip}
\newtheorem{theorem}{Theorem}[section]
\newtheorem{lemma}[theorem]{Lemma}
\newtheorem{proposition}[theorem]{Proposition}
\newtheorem{corollary}[theorem]{Corollary}

\newcommand{\F}{\mathbb{F}}

\title{A Delsarte Linear Programming Approach to the Erd\H{o}s--Falconer Distance Problem over Finite Fields}
\author{Tao Zhang\\
\footnotesize Institute of Mathematics and Interdisciplinary Sciences, Xidian University, Xi'an 710126, China.\\
	}
\date{}

\begin{document}
\maketitle
\begingroup
\renewcommand\thefootnote{}% 去掉编号格式
\footnotetext{E-mail addresses: zhant220@163.com (T. Zhang).}
\addtocounter{footnote}{-1}% 不影响后续脚注编号
\endgroup

\begin{abstract}
	We introduce a Delsarte linear programming approach to the finite field Erd\H{o}s--Falconer distance problem. Let \(q\) be an odd prime power, let \(n\) be even, and let \(Q\) be a non-degenerate quadratic form on \(\mathbb{F}_q^n\). For \(E\subset \mathbb{F}_q^n\), define
	\[
	\Delta_Q(E)=\{Q(x-y):\ x,y\in E\}.
	\]
	We prove that, for every fixed \(0<\alpha<\frac{1}{2}\), there exist constants \(C_\alpha>0\) and \(q_\alpha\) such that if \(q\ge q_\alpha\) and $|E|\ge C_\alpha q^{\frac n2+\frac13},$
	then
	\[
	|\Delta_Q(E)|>1+\alpha(q-1).
	\]
	In particular, \(\Delta_Q(E)\) contains a positive proportion of the elements of \(\mathbb{F}_q\), and hence \(|\Delta_Q(E)|\gg q\).
	
	Our result applies uniformly to all non-degenerate quadratic forms in even-dimensional finite field vector spaces. In the Euclidean case
	\[
	Q(x)=x_1^2+\cdots+x_n^2,
	\]
	it improves, for every even \(n\ge 4\) over arbitrary finite fields, the general exponent \(\frac{n+1}{2}\) obtained by Iosevich and Rudnev to $\frac n2+\frac13.$
	The proof is based on the association scheme arising from the level sets of \(Q\). By analyzing the corresponding eigenvalues through Gauss sums and Kloosterman sums, we construct a suitable feasible solution to the Delsarte linear program. This provides a new algebraic-combinatorial method for obtaining distance set estimates over finite fields.
	
		\medskip
	
	
	\smallskip
	
\end{abstract}

\section{Introduction}

For a finite set \(E\subset \mathbb{R}^n\), the classical Erd\H{o}s distance problem asks how small the number of distinct distances determined by \(E\) can be in terms of \(|E|\). A natural continuous analogue is the Falconer distance problem. For a compact set \(E\subset\mathbb{R}^n\), define its Euclidean distance set by
\[
\Delta_{\mathbb{R}^n}(E)
=
\left\{
\sum_{i=1}^n (x_i-y_i)^2:\ x,y\in E
\right\}.
\]
The Falconer conjecture asserts that if $\dim_H(E)>\frac n2,$
then \(\Delta_{\mathbb{R}^n}(E)\) has positive Lebesgue measure. This threshold is expected to be sharp.

Substantial progress has been made on this problem. In the plane, Guth, Iosevich, Ou and Wang \cite{GIOW2020} proved that \(|\Delta_{\mathbb{R}^2}(E)|>0\) whenever $\dim_H(E)>\frac54.$
In higher dimensions, Du, Ou, Ren and Zhang \cite{DORZ23} showed that for \(n\ge 3\), the distance set has positive Lebesgue measure provided $\dim_H(E)>
\frac n2+\frac14-\frac{1}{8n+4}.$
These results illustrate the depth of the distance problem and its close connection with harmonic analysis, geometric measure theory, and incidence geometry.

Finite field analogues of classical problems in harmonic analysis, geometric measure theory, and combinatorics have attracted considerable attention in recent years. On the one hand, the discrete setting often makes the underlying mechanisms more transparent; on the other hand, finite fields introduce new arithmetic phenomena that have no direct analogue over the real numbers. We refer the reader to \cite{BKT2004,IR2007,MT2004,W1999} and the references therein for related developments. In this paper, we study the finite field analogue of the Erd\H{o}s--Falconer distance problem.

Let \(\mathbb{F}_q\) be a finite field with \(q\) elements, where \(q\) is an odd prime power, and let \(\mathbb{F}_q^n\) denote the \(n\)-dimensional vector space over \(\mathbb{F}_q\). For \(x=(x_1,\dots,x_n)\in \mathbb{F}_q^n\), define
\[
d(x)=\sum_{i=1}^n x_i^2.
\]
For a set \(E\subset \mathbb{F}_q^n\), its distance set is
\[
\Delta_d(E)
=
\{d(x-y):\ x,y\in E\}.
\]
The finite field Erd\H{o}s--Falconer distance problem asks for the smallest exponent \(\alpha_n\) such that if $|E|\gg q^{\alpha_n}$, then $|\Delta_d(E)|\gg q.$

A foundational result of Iosevich and Rudnev \cite{IR2007}, proved via Fourier analytic methods, gives $\alpha_n\le \frac{n+1}{2}.$
This bound is sharp in odd dimensions \cite{HIKR2011}. In even dimensions, however, the situation is much more delicate. The known examples only imply the lower bound $\alpha_n\ge \frac n2,$
and it has been conjectured that this lower bound is the correct threshold for all even \(n\).

Despite considerable effort \cite{CEHIK2012,CGKPTZ25,F24,KS2015,KPSV2021,KPV2021,KPP2023,MPPRS2022,PS2020,PVZ2019}, the conjecture in even dimensions remains open. In dimension \(2\), the best known exponent over arbitrary finite fields is \(\frac{4}{3}\) \cite{BHIPR2017,CEHIK2012}, while over prime fields the exponent has been improved to \(\frac{5}{4}\) \cite{MPPRS2022}. In dimension \(4\), Pham and Xue \cite{PX26} recently improved the exponent over prime fields from \(\frac{5}{2}\) to \(\frac{77}{31}\). These developments show that any improvement over the general exponent \(\frac{n+1}{2}\), especially uniformly in all even dimensions and over arbitrary finite fields, is highly nontrivial.

The main purpose of this paper is to establish such a uniform improvement. Moreover, instead of restricting ourselves to the standard Euclidean quadratic form \(d(x)\), we work in a more general setting. Let
\[
Q(x)=x^T A x
\]
be a non-degenerate quadratic form on \(\mathbb{F}_q^{2m}\), where \(A\) is symmetric and invertible. For \(E\subset \mathbb{F}_q^{2m}\), define
\[
\Delta_Q(E)
=
\{Q(x-y):\ x,y\in E\}.
\]
Our main theorem shows that, for every non-degenerate quadratic form in even dimension, a set of size \(q^{m+\frac{1}{3}}\) already determines a positive proportion of all possible distances.

\begin{theorem}\label{mainthm}
	Fix \(0<\alpha<\frac12\). There exist constants \(C_\alpha>0\) and \(q_\alpha\) such that, for every odd prime power \(q\ge q_\alpha\), every integer \(m\ge 1\), every non-degenerate quadratic form \(Q\) on \(\mathbb{F}_q^{2m}\), and every set \(E\subset\mathbb{F}_q^{2m}\) with $|E|\ge C_\alpha q^{m+\frac13},$
	we have
	\[
	|\Delta_Q(E)|>1+\alpha(q-1).
	\]
\end{theorem}

Thus the distance set contains more than an arbitrary fixed proportion \(\alpha<\frac{1}{2}\) of the nonzero field elements, up to the additional possible value \(0\). In particular, Theorem~\ref{mainthm} immediately implies the following consequence.

\begin{corollary}
	Let \(E\subset\mathbb{F}_q^{2m}\). If $|E|\ge C q^{m+\frac13}$
	for a sufficiently large constant \(C\), then
	\[
	|\Delta_Q(E)|\gg q.
	\]
	In particular, for the standard quadratic form $d(x)=\sum_{i=1}^{2m}x_i^2,$
	if \(E\subset\mathbb{F}_q^{2m}\) satisfies $|E|\ge C q^{m+\frac13},$
	then
	\[
	|\Delta_d(E)|\gg q.
	\]
\end{corollary}

In terms of the ambient dimension \(n=2m\), our result gives the exponent $\frac n2+\frac13$
for all even dimensions. This improves the general exponent \(\frac{n+1}{2}\) of Iosevich and Rudnev \cite{IR2007} by \(\frac{1}{6}\) in the exponent. In dimension \(2\), it recovers the best known exponent \(\frac{4}{3}\) over arbitrary finite fields. In every even dimension \(n\ge 4\), it gives a genuine improvement over the previously available general bound. In particular, in dimension \(4\), our exponent is \(\frac{7}{3}\), improving upon the exponent \(\frac{5}{2}\) over arbitrary finite fields and also strengthening the recent prime-field exponent \(\frac{77}{31}\) of \cite{PX26}.

The proof is based on the Delsarte linear programming method, originally
introduced in Delsarte's algebraic approach to association schemes in coding
theory \cite{D1973}. This approach brings tools from algebraic
combinatorics into the finite field distance problem and differs from the
Fourier analytic and incidence-geometric methods that have traditionally
been used in this area. More precisely, we analyze the association scheme
naturally induced by the level sets of the quadratic form \(Q\); for standard
background on association schemes, see \cite{BCN1989}.
The relevant spectral information is encoded by Gauss sums, Kloosterman sums, and the associated Kloosterman matrix. By constructing a suitable feasible solution to the corresponding linear program, we obtain an upper bound for the size of a set whose distance set omits many field elements. This yields Theorem~\ref{mainthm}.

The paper is organized as follows. In Section~\ref{pre}, we recall the necessary facts about Gauss sums, Kloosterman sums, and the Kloosterman matrix. In Section~\ref{main}, we apply the Delsarte linear programming bound to prove the main theorem.

\section{Preliminaries}\label{pre}

In this section we collect several elementary facts that will be used in the proof of the main theorem. Throughout the paper, \(q\) is an odd prime power, \(\chi\) denotes a fixed nontrivial additive character of \(\mathbb{F}_q\), and \(\eta\) denotes the quadratic character of \(\mathbb{F}_q\), extended by setting \(\eta(0)=0\).

We first recall some standard facts about Gauss sums. Let
\[
G(\chi):=\sum_{x\in\mathbb{F}_q}\chi(x^2).
\]
Then
\[
G(\chi)^2=\eta(-1)q,
\]
and, for every \(t\in\mathbb{F}_q^\times\),
\[
\sum_{x\in\mathbb{F}_q}\chi(tx^2)=\eta(t)G(\chi).
\]

The following lemma records the exponential sum associated with a non-degenerate quadratic form in even dimension.

\begin{lemma}\label{lem-gauss-qf}
	Let \(Q(x)=x^T A x\) be a non-degenerate quadratic form on \(\mathbb{F}_q^{2m}\), where \(A\) is symmetric and invertible. Then, for every \(t\in\mathbb{F}_q^\times\), $\sum_{x\in\mathbb{F}_q^{2m}}\chi(tQ(x))
	=
	\eta(\det A)\eta(-1)^m q^m.$
\end{lemma}

\begin{proof}
	Since the characteristic of \(\mathbb{F}_q\) is odd, the quadratic form \(Q\) can be diagonalized. Thus, after an invertible linear change of variables, we may write $Q(x)=\sum_{i=1}^{2m}a_i x_i^2$
	with \(a_i\in\mathbb{F}_q^\times\). Moreover, \(\prod_i a_i\) differs from \(\det(A)\) by a square factor, and therefore $\eta\left(\prod_{i=1}^{2m}a_i\right)=\eta(\det A).$
	It follows that
	\begin{align*}
		\sum_{x\in\mathbb{F}_q^{2m}}\chi(tQ(x))
		&=
		\prod_{i=1}^{2m}
		\sum_{x_i\in\mathbb{F}_q}\chi(ta_i x_i^2)  \\
		&=
		\prod_{i=1}^{2m}\eta(ta_i)G(\chi)  \\
		&=
		\eta(t)^{2m}
		\eta\left(\prod_{i=1}^{2m}a_i\right)
		G(\chi)^{2m}  \\
		&=
		\eta(\det A)\eta(-1)^m q^m.
	\end{align*}
	This proves the lemma.
\end{proof}

We next recall some basic properties of Kloosterman sums. For \(a\in\mathbb{F}_q\), define the classical Kloosterman sum
\[
K(a)=\sum_{x\in\mathbb{F}_q^\times}\chi\left(x+\frac{a}{x}\right).
\]
Notice that \(K(a)\) is real, since $\overline{K(a)}
=
\sum_{x\in\mathbb{F}_q^\times}
\chi\left(-x-\frac{a}{x}\right)
=
K(a).$
We define the Kloosterman matrix
\[
K=(K(st))_{s,t\in\mathbb{F}_q^\times},
\]
whose rows and columns are indexed by \(\mathbb{F}_q^\times\). We also write \(\mathbf{1}\) for the all-one vector indexed by \(\mathbb{F}_q^\times\), and \(J\) for the all-one matrix.

\begin{lemma}\label{lem-1}
	We have $	K\mathbf{1}=\mathbf{1},\
	K^2=q^2I-(q+1)J.$
\end{lemma}

\begin{proof}
	For \(s\in\mathbb{F}_q^\times\), we have
	\[\sum_{t\in\mathbb{F}_q^\times}K(st)
	=
	\sum_{x\in\mathbb{F}_q^\times}
	\chi(x)
	\sum_{t\in\mathbb{F}_q^\times}
	\chi\left(\frac{st}{x}\right)  \\
	=
	-\sum_{x\in\mathbb{F}_q^\times}\chi(x)
	=
	1.\]
	Thus \(K\mathbf{1}=\mathbf{1}\).
	
	Next, for \(s,u\in\mathbb{F}_q^\times\),
	\begin{align*}
		(K^2)_{s,u}
		&=
		\sum_{t\in\mathbb{F}_q^\times}K(st)K(tu)  \\
		&=
		\sum_{x,y\in\mathbb{F}_q^\times}
		\chi(x+y)
		\sum_{t\in\mathbb{F}_q^\times}
		\chi\left(t\left(\frac{s}{x}+\frac{u}{y}\right)\right)  \\
		&=
		q
		\sum_{\substack{x,y\in\mathbb{F}_q^\times\\ \frac{s}{x}+\frac{u}{y}=0}}
		\chi(x+y)
		-
		\sum_{x,y\in\mathbb{F}_q^\times}\chi(x+y)  \\
		&=
		q\sum_{x\in\mathbb{F}_q^\times}
		\chi\left(x-\frac{ux}{s}\right)
		-1.
	\end{align*}
	If \(s=u\), this is \(q(q-1)-1=q^2-q-1\). If \(s\ne u\), then $\sum_{x\in\mathbb{F}_q^\times}
	\chi\left(x-\frac{ux}{s}\right)=-1,$
	and hence \((K^2)_{s,u}=-q-1\). Therefore $	K^2=q^2I-(q+1)J.$
\end{proof}

We shall also need the following cubic estimate for Kloosterman sums.

\begin{lemma}\label{lem-2}
	For \(t_1,t_2,t_3\in\mathbb{F}_q^\times\), we have $	\left|
	\sum_{s\in\mathbb{F}_q^\times}
	K(st_1)K(st_2)K(st_3)
	\right|
	\le (q+1)^2.$
\end{lemma}

\begin{proof}
	Expanding the three Kloosterman sums and summing first in \(s\), we obtain
	\begin{align*}
		&\sum_{s\in\mathbb{F}_q^\times}
		K(st_1)K(st_2)K(st_3)  \\
		&=
		\sum_{s\in\mathbb{F}_q^\times}
		\sum_{x_1,x_2,x_3\in\mathbb{F}_q^\times}
		\chi\left(
		x_1+x_2+x_3
		+s\left(\frac{t_1}{x_1}+\frac{t_2}{x_2}+\frac{t_3}{x_3}\right)
		\right)  \\
		&=
		1+
		q
		\sum_{\substack{x_1,x_2,x_3\in\mathbb{F}_q^\times\\
				\frac{t_1}{x_1}+\frac{t_2}{x_2}+\frac{t_3}{x_3}=0}}
		\chi(x_1+x_2+x_3).
	\end{align*}
	To evaluate the remaining sum, write \(x_1=rx_2\), where \(r\in\mathbb{F}_q^\times\). The constraint becomes $\frac{t_1}{rx_2}+\frac{t_2}{x_2}+\frac{t_3}{x_3}=0.$
	Thus \(r\ne -t_1/t_2\), and $x_3=-\frac{t_3rx_2}{t_1+t_2r}.$
	Consequently,
	\begin{align*}
		&\sum_{\substack{x_1,x_2,x_3\in\mathbb{F}_q^\times\\
				\frac{t_1}{x_1}+\frac{t_2}{x_2}+\frac{t_3}{x_3}=0}}
		\chi(x_1+x_2+x_3)  \\
		&=
		\sum_{r\in\mathbb{F}_q^\times\setminus\{-\frac{t_1}{t_2}\}}
		\sum_{x_2\in\mathbb{F}_q^\times}
		\chi\left(
		x_2\left(
		r+1-\frac{t_3r}{t_1+t_2r}
		\right)
		\right).
	\end{align*}
	Let \(N\) be the number of \(r\in\mathbb{F}_q^\times\setminus\{-\frac{t_1}{t_2}\}\) for which $r+1-\frac{t_3r}{t_1+t_2r}=0.$
	Equivalently, \(N\) is the number of roots in
	\(\mathbb{F}_q^\times\setminus\{-\frac{t_1}{t_2}\}\) of the quadratic polynomial $t_2r^2+(t_1+t_2-t_3)r+t_1.$
	Hence \(N\le2\). For those \(r\), the inner sum over \(x_2\) equals \(q-1\); for all other admissible \(r\), it equals \(-1\). Therefore
	\[		\sum_{s\in\mathbb{F}_q^\times}
	K(st_1)K(st_2)K(st_3)
	=
	1+q\bigl(N(q-1)-(q-2-N)\bigr) 
	=
	q^2(N-1)+2q+1.\]
	Since \(N\le2\), the absolute value is at most \((q+1)^2\).
\end{proof}

We now introduce a probabilistic notation that will be convenient in the linear programming argument. A probability vector \(\mu\) on a finite set \(S\) is a vector satisfying
\[
\mu(s)\ge0
\quad\text{for all }s\in S,
\qquad
\sum_{s\in S}\mu(s)=1.
\]
For a function \(f\) on \(S\), we write
\[
\mathbb{E}_S f=\frac{1}{|S|}\sum_{s\in S}f(s)
\]
for the normalized average. When the underlying set is clear, we simply write \(\mathbb{E}f\).

\begin{lemma}\label{lem-3}
	Let \(\mu\) be a probability vector on \(\mathbb{F}_q^\times\), and define $	F(s)=(K\mu)(s)
	=
	\sum_{t\in\mathbb{F}_q^\times}K(st)\mu(t).$
	Then
	\begin{align*}
		&\sum_{s\in\mathbb{F}_q^\times}F(s)=1, \\
		&\sum_{s\in\mathbb{F}_q^\times}F(s)^2
		=
		q^2\sum_{t\in\mathbb{F}_q^\times}\mu(t)^2-q-1, \\
		&\left|
		\sum_{s\in\mathbb{F}_q^\times}F(s)^3
		\right|
		\le (q+1)^2.
	\end{align*}
	Moreover, if $X=F-\frac{1}{q-1}$ or $X=\frac{1}{q-1}-F,$
	then $|\mathbb{E}X^3|\le 6q.$
\end{lemma}

\begin{proof}
	The first identity follows from \(K\mathbf{1}=\mathbf{1}\), or equivalently from the fact that every column sum of \(K\) is equal to \(1\).
	
	For the second identity, since \(K\) is symmetric, Lemma~\ref{lem-1} gives
	\[
	\sum_{s\in\mathbb{F}_q^\times}F(s)^2
	=
	\mu^T K^2\mu
	=
	q^2\sum_{t\in\mathbb{F}_q^\times}\mu(t)^2
	-(q+1)\left(\sum_{t\in\mathbb{F}_q^\times}\mu(t)\right)^2=	q^2\sum_{t\in\mathbb{F}_q^\times}\mu(t)^2-q-1.
	\]
	
	For the cubic estimate, we expand \(F(s)^3\) and apply Lemma~\ref{lem-2}:
	\begin{align*}
		\left|
		\sum_{s\in\mathbb{F}_q^\times}F(s)^3
		\right|
		&=
		\left|
		\sum_{t_1,t_2,t_3\in\mathbb{F}_q^\times}
		\mu(t_1)\mu(t_2)\mu(t_3)
		\sum_{s\in\mathbb{F}_q^\times}
		K(st_1)K(st_2)K(st_3)
		\right|  \\
		&\le
		(q+1)^2
		\sum_{t_1,t_2,t_3\in\mathbb{F}_q^\times}
		\mu(t_1)\mu(t_2)\mu(t_3)  \\
		&=
		(q+1)^2.
	\end{align*}
	
	It remains to prove the final estimate. Since $\mathbb{E}F=\frac{1}{q-1},$ $\mathbb{E}F^2
	\le
	\frac{q^2}{q-1},$ and $	|\mathbb{E}F^3|
	\le
	\frac{(q+1)^2}{q-1},$
	we have
	\[	\left|
	\mathbb{E}\left(F-\frac{1}{q-1}\right)^3
	\right|
	\le
	\frac{(q+1)^2}{q-1}
	+
	\frac{3q^2}{(q-1)^2}
	+
	\frac{2}{(q-1)^3}  
	\le 6q.\]
	The estimate for \(X=\frac{1}{q-1}-F\) follows by changing sign.
\end{proof}

The next elementary lemma converts information about the third moment into a pointwise lower bound. It will be used to force \(K\mu\) to have both a substantially negative value and a substantially positive value when \(\mu\) has small support.

\begin{lemma}\label{lem-4}
	Let \(0<\alpha<\frac12\). Let \(X\) be a real-valued function on a finite set \(\Omega\), with $\mathbb{E}_\Omega X=0.$
	Suppose that \(A\ge0\), that $	X\ge -A$
	on \(\Omega\), and that $X=-A$
	on at least \((1-\alpha)|\Omega|\) points. If $|\mathbb{E}_\Omega X^3|\le M,$
	then $	A^3\frac{(1-\alpha)(1-2\alpha)}{\alpha^2}
	\le M.$
\end{lemma}

\begin{proof}
	Let $Z=\{\omega\in\Omega:\ X(\omega)=-A\},$
	and put $k=|Z|,$ $\ell=|\Omega|-k.$
	If \(\ell=0\), then \(X\equiv -A\). Since \(\mathbb{E}_\Omega X=0\), we have \(A=0\), and the result is trivial. We may therefore assume that \(\ell>0\).
	
	Since \(\mathbb{E}_\Omega X=0\), we have $\sum_{\omega\notin Z}X(\omega)=kA.$
	Let $V=\mathbb{E}_\Omega X^2.$
	By the Cauchy--Schwarz inequality,
	\[
	|\Omega|V
	=
	\sum_{\omega\in\Omega}X(\omega)^2
	\ge
	kA^2+\frac{k^2A^2}{\ell}
	=
	\frac{|\Omega|kA^2}{\ell}.
	\]
	Thus $V\ge rA^2,$ where $r:=\frac{k}{\ell}.$
	Since \(k\ge(1-\alpha)|\Omega|\), we have $	r=\frac{k}{\ell}\ge\frac{1-\alpha}{\alpha}>1.$
	
	If \(A=0\), there is nothing to prove. Assume \(A>0\). Since \(X+A\ge0\), we have $(X+A)\left(X-\frac{V}{A}\right)^2\ge0.$
	Taking normalized averages and using \(\mathbb{E}_\Omega X=0\), we obtain
	\[
	\mathbb{E}_\Omega X^3
	\ge
	\frac{V^2}{A}-AV=A^3\frac{V}{A^2}\left(\frac{V}{A^2}-1\right)
	\ge
	A^3 r(r-1).
	\]
	Note that $r(r-1)
	\ge
	\frac{1-\alpha}{\alpha}
	\left(
	\frac{1-\alpha}{\alpha}-1
	\right)
	=
	\frac{(1-\alpha)(1-2\alpha)}{\alpha^2}.$
	Since \(|\mathbb{E}_\Omega X^3|\le M\), the desired estimate follows.
\end{proof}

We now derive the key consequence for probability vectors with small support. In the following statement, the constants depend only on \(\alpha\).

\begin{lemma}\label{lem-5}
	Let \(0<\alpha<\frac12\). There exist constants \(c_\alpha>0\) and \(q_\alpha\) such that, for every odd prime power \(q\ge q_\alpha\) and every probability vector \(\mu\) on \(\mathbb{F}_q^\times\) satisfying $|\operatorname{supp}(\mu)|\le \alpha(q-1),$
	we have $-\min_{s\in\mathbb{F}_q^\times}(K\mu)(s)
	\ge
	c_\alpha q^{-\frac{1}{3}}$
	and $\max_{s\in\mathbb{F}_q^\times}(K\mu)(s)
	\ge
	c_\alpha q^{-\frac{1}{3}}.$
\end{lemma}

\begin{proof}
	We prove the negative estimate first. Put
	\[
	F=K\mu,
	\qquad
	m_-=\max\left\{0,-\min_{s\in\mathbb{F}_q^\times}F(s)\right\}.
	\]
	Then $F+m_-\mathbf{1}\ge0,$
	and hence $\nu_-=
	\frac{F+m_-\mathbf{1}}{1+(q-1)m_-}$
	is a probability vector on \(\mathbb{F}_q^\times\). Let $D_-=1+(q-1)m_-.$
	Using Lemma~\ref{lem-1}, we get
	\[		K\nu_-
	=
	\frac{K^2\mu+m_-K\mathbf{1}}{D_-}  
	=
	\frac{q^2\mu-(q+1-m_-)\mathbf{1}}{D_-}.\]

	Thus \(K\nu_-\ge -B_-\) everywhere, and equality holds outside \(\operatorname{supp}(\mu)\), where $B_-=\frac{q+1-m_-}{D_-}.$
	
	If \(m_-\ge q^{-\frac{1}{3}}\), then the desired estimate follows after choosing \(c_\alpha\le1\). We may therefore assume that $m_-<q^{-\frac{1}{3}}.$
	For \(q\) sufficiently large, this implies \(q+1-m_-\ge q\).
	
	Now define
	\[
	X_-=K\nu_- -\frac{1}{q-1}.
	\]
	Then \(\mathbb{E}X_-=0\), and by Lemma~\ref{lem-3}, $|\mathbb{E}X_-^3|\le 6q.$
	Moreover, $X_-\ge -A_-,$ and $A_-:=B_-+\frac{1}{q-1},$
	and equality holds outside \(\operatorname{supp}(\mu)\). Since $|\operatorname{supp}(\mu)|\le \alpha(q-1),$
	the equality \(X_-=-A_-\) holds on at least \((1-\alpha)(q-1)\) points. Applying Lemma~\ref{lem-4}, we obtain
	\[
	A_-^3
	\frac{(1-\alpha)(1-2\alpha)}{\alpha^2}
	\le
	6q.
	\]
	Thus $A_-\ll_\alpha q^{\frac{1}{3}}.$
	In particular, $\frac{q+1-m_-}{D_-}\ll_\alpha q^{\frac{1}{3}}.$
	Since \(q+1-m_-\ge q\), it follows that $D_-\gg_\alpha q^{\frac{2}{3}}.$
	Therefore
	\[
	m_-=\frac{D_- -1}{q-1}\gg_\alpha q^{-\frac{1}{3}}.
	\]
	This proves the first estimate.
	
	We now prove the positive estimate. Let
	\[
	m_+=\max_{s\in\mathbb{F}_q^\times}F(s).
	\]
	If \(m_+\ge q^{-\frac{1}{3}}\), then the desired estimate is immediate. Hence we may assume that $m_+<q^{-\frac{1}{3}}.$
	Since \(\mathbb{E}F=\frac{1}{q-1}\), we have \(m_+\ge\frac{1}{q-1}\). Equality would imply \(F\equiv\frac{1}{q-1}\). Applying \(K\) and using Lemma~\ref{lem-1}, this would force \(\mu\) to be the uniform probability vector on \(\mathbb{F}_q^\times\), contradicting $|\operatorname{supp}(\mu)|\le\alpha(q-1)$
	because \(\alpha<\frac{1}{2}\). Hence $D_+:=(q-1)m_+-1>0.$
	Define $\nu_+=\frac{m_+\mathbf{1}-F}{D_+}.$
	Then \(\nu_+\) is a probability vector. Again using Lemma~\ref{lem-1}, we have
	\[K\nu_+
	=
	\frac{m_+K\mathbf{1}-K^2\mu}{D_+}  
	=
	\frac{(q+1+m_+)\mathbf{1}-q^2\mu}{D_+}.\]
	Thus \(K\nu_+\le B_+\) everywhere, and equality holds outside \(\operatorname{supp}(\mu)\), where $B_+=\frac{q+1+m_+}{D_+}.$
	
	Let
	\[
	X_+=\frac{1}{q-1}-K\nu_+.
	\]
	Then \(\mathbb{E}X_+=0\), and Lemma~\ref{lem-3} gives $|\mathbb{E}X_+^3|\le 6q.$
	Moreover, $	X_+\ge -A_+,$ and $A_+:=B_+-\frac{1}{q-1},$
	and equality holds outside \(\operatorname{supp}(\mu)\). Applying Lemma~\ref{lem-4} again yields $A_+\ll_\alpha q^{\frac{1}{3}}.$
	Consequently, $B_+\ll_\alpha q^{\frac{1}{3}}.$
	Since \(q+1+m_+\ge q\), we obtain $D_+\gg_\alpha q^{\frac{2}{3}}.$
	Therefore
	\[
	m_+=\frac{D_+ +1}{q-1}\gg_\alpha q^{-\frac{1}{3}}.
	\]
	This proves the second estimate, and hence the lemma follows.
\end{proof}

The last ingredient is a standard minimax lemma, which will be used to turn the pointwise estimates above into uniform linear programming certificates.

\begin{lemma}[von Neumann minimax lemma]\cite{V1928}\label{lem-6}
	Let \(S,T\) be finite sets, let \((b_{s,t})_{s\in S,t\in T}\) be a real matrix, and let \(\gamma\in\mathbb{R}\). Suppose that, for every probability vector \(\nu\) on \(T\),
	\[
	\max_{s\in S}\sum_{t\in T}b_{s,t}\nu(t)\ge \gamma.
	\]
	Then there exists a probability vector \(\mu\) on \(S\) such that
	\[
	\sum_{s\in S}\mu(s)b_{s,t}\ge \gamma
	\]
	for all \(t\in T\).
\end{lemma}

Combining Lemma~\ref{lem-5} with the minimax lemma gives the following uniform form, which will be used in the proof of the main theorem.

\begin{lemma}\label{lem-7}
	Let \(0<\alpha<\frac12\). There exist constants \(c_\alpha>0\) and \(q_\alpha\) such that, whenever \(q\ge q_\alpha\) and
	\[
	T\subset\mathbb{F}_q^\times,
	\qquad
	1\le |T|\le \alpha(q-1),
	\]
	there exist probability vectors \(\mu_-\) and \(\mu_+\) on \(\mathbb{F}_q^\times\) such that
	\[
	\sum_{s\in\mathbb{F}_q^\times}\mu_-(s)K(st)
	\le
	-c_\alpha q^{-\frac{1}{3}}
	\]
	for all \(t\in T\), and
	\[
	\sum_{s\in\mathbb{F}_q^\times}\mu_+(s)K(st)
	\ge
	c_\alpha q^{-\frac{1}{3}}
	\]
	for all \(t\in T\).
\end{lemma}

\begin{proof}
	Let \(\nu\) be a probability vector on \(T\), and extend it by zero to a probability vector on \(\mathbb{F}_q^\times\). Since $|\operatorname{supp}(\nu)|\le |T|\le \alpha(q-1),$
	Lemma~\ref{lem-5} gives $	\max_{s\in\mathbb{F}_q^\times}
	\sum_{t\in T}K(st)\nu(t)
	\ge
	c_\alpha q^{-\frac{1}{3}}.$
	Applying Lemma~\ref{lem-6} with $S=\mathbb{F}_q^\times,$ and $b_{s,t}=K(st),$
	we obtain a probability vector \(\mu_+\) on \(\mathbb{F}_q^\times\) such that
	\[
	\sum_{s\in\mathbb{F}_q^\times}\mu_+(s)K(st)
	\ge
	c_\alpha q^{-\frac{1}{3}}
	\]
	for every \(t\in T\).
	
	Similarly, Lemma~\ref{lem-5} gives
	\[
	\max_{s\in\mathbb{F}_q^\times}
	\sum_{t\in T}(-K(st))\nu(t)
	=
	-\min_{s\in\mathbb{F}_q^\times}
	\sum_{t\in T}K(st)\nu(t)
	\ge
	c_\alpha q^{-\frac{1}{3}}.
	\]
	Applying Lemma~\ref{lem-6} with $b_{s,t}=-K(st),$
	we obtain a probability vector \(\mu_-\) on \(\mathbb{F}_q^\times\) such that
	\[
	\sum_{s\in\mathbb{F}_q^\times}\mu_-(s)K(st)
	\le
	-c_\alpha q^{-\frac{1}{3}}
	\]
	for every \(t\in T\). This completes the proof.
\end{proof}

\section{Proof of Theorem~\ref{mainthm}}\label{main}

In this section we prove the main theorem. We begin by deriving the
Delsarte linear programming constraints associated with the quadratic form
\(Q\).

Let
\[
\varepsilon_Q:=\eta(\det A)\eta(-1)^m.
\]
Thus \(\varepsilon_Q\in\{\pm1\}\). For \(u\in\F_q\), let
\[
S_u=\{x\in\F_q^{2m}:\ Q(x)=u\}.
\]
The diagonal relation, the punctured zero level set \(S_0\setminus\{0\}\),
and the nonzero level sets \(S_t\), \(t\in\F_q^\times\), form the natural
translation association scheme attached to \(Q\). The dual scheme is indexed
by the level sets of the dual quadratic form
\[
Q^\perp(z)=\frac14 z^T A^{-1}z.
\]
The following quantities are the normalized inner distribution of \(E\) with
respect to this association scheme.

Let \(E\subset\F_q^{2m}\). Define
\begin{align*}
	a_0
	&=
	\frac{1}{|E|}
	\left|
	\{(x,y)\in E^2:\ x\ne y,\ Q(x-y)=0\}
	\right|,\\
	a_t
	&=
	\frac{1}{|E|}
	\left|
	\{(x,y)\in E^2:\ x\ne y,\ Q(x-y)=t\}
	\right|,
	\qquad t\in\F_q^\times.
\end{align*}
Let
\[
T=\{t\in\F_q^\times:\ a_t>0\}.
\]
Since the diagonal contributes \(1\) to the normalized inner distribution, we
have
\[
|E|=1+a_0+\sum_{t\in T}a_t.
\]

The next proposition gives the nonzero dual-level constraints in the
Delsarte linear program. Equivalently, it says that the MacWilliams transform
of the inner distribution is nonnegative on every nonzero dual level set.
The Kloosterman sums appearing below are precisely the eigenvalues of the
level-set association scheme on these dual levels.

\begin{proposition}\label{prop-lp-constraints}
	For every \(s\in\F_q^\times\), we have $q^m+\varepsilon_Q
	\left(
	-1-a_0+\sum_{t\in T}K(st)a_t
	\right)
	\ge0.$
\end{proposition}

\begin{proof}
	Let $\widehat{E}(z):=\sum_{x\in E}\chi(z\cdot x).$
	For each \(s\in\F_q^\times\), positivity of the Fourier mass on the dual
	level set \(Q^\perp(z)=s\) gives
	\[
	0
	\le
	\sum_{Q^\perp(z)=s}|\widehat{E}(z)|^2.
	\]
	We now compute this quantity explicitly. By orthogonality of additive
	characters,
	\begin{align*}
		0
		&\le
		\sum_{Q^\perp(z)=s}
		\sum_{x,y\in E}\chi((x-y)\cdot z) \\
		&=
		\frac1q
		\sum_{x,y\in E}
		\sum_{z\in\F_q^{2m}}
		\chi((x-y)\cdot z)
		\sum_{\lambda\in\F_q}
		\chi\bigl(\lambda(Q^\perp(z)-s)\bigr) \\
		&=
		\frac1q
		\sum_{\lambda\in\F_q}
		\chi(-\lambda s)
		\sum_{x,y\in E}
		\sum_{z\in\F_q^{2m}}
		\chi\bigl(\lambda Q^\perp(z)+(x-y)\cdot z\bigr).
	\end{align*}
	The contribution of \(\lambda=0\) is $q^{2m-1}|E|.$
	For \(\lambda\ne0\), completing the square gives
	\[
	\lambda Q^\perp(z)+(x-y)\cdot z
	=
	\lambda Q^\perp\left(z+\frac{2A(x-y)}{\lambda}\right)
	-
	\frac{Q(x-y)}{\lambda}.
	\]
	Hence, using Lemma~\ref{lem-gauss-qf}, we obtain
	\begin{align*}
		0
		&\le
		q^{2m-1}|E|
		+
		\varepsilon_Q q^{m-1}
		\sum_{\lambda\in\F_q^\times}
		\chi(-\lambda s)
		\sum_{x,y\in E}
		\chi\left(-\frac{Q(x-y)}{\lambda}\right) \\
		&=
		q^{2m-1}|E|
		+
		\varepsilon_Q q^{m-1}
		\sum_{x,y\in E}
		K(sQ(x-y)).
	\end{align*}
	
	We now express the last sum in terms of the inner distribution of \(E\).
	Since $K(0)=\sum_{r\in\F_q^\times}\chi(r)=-1,$
	the pairs with \(Q(x-y)=0\), including the diagonal pairs, contribute $-|E|(1+a_0).$
	The pairs with \(Q(x-y)=t\in T\) contribute $|E|K(st)a_t.$
	Therefore $\sum_{x,y\in E}K(sQ(x-y))
	=
	|E|
	\left(
	-1-a_0+\sum_{t\in T}K(st)a_t
	\right).$
	Substituting this into the previous inequality yields
	\[
	0
	\le
	q^{2m-1}|E|
	+
	\varepsilon_Q q^{m-1}|E|
	\left(
	-1-a_0+\sum_{t\in T}K(st)a_t
	\right).
	\]
	Dividing by \(q^{m-1}|E|\) gives the desired inequality.
\end{proof}

When \(\varepsilon_Q=-1\), we shall also need the constraint coming from the
zero dual level. This additional estimate controls the possible size of the
punctured zero-distance part of the inner distribution.

\begin{lemma}\label{lem-9}
	Assume that $\eta(\det A)=-\eta(-1)^m,$
	or equivalently \(\varepsilon_Q=-1\). Then $a_0\le q^{m-1}+\frac{|E|}{q}.$
\end{lemma}

\begin{proof}
	We apply the same Fourier positivity argument on the zero dual level
	\(Q^\perp(z)=0\). Since \(\varepsilon_Q=-1\), we get
	\[0
	\le
	\sum_{Q^\perp(z)=0}|\widehat{E}(z)|^2 
	=
	q^{2m-1}|E|
	-
	q^{m-1}
	\sum_{\lambda\in\F_q^\times}
	\sum_{x,y\in E}
	\chi\left(-\frac{Q(x-y)}{\lambda}\right).\]
	For fixed \(x,y\), we have
	\[
	\sum_{\lambda\in\F_q^\times}
	\chi\left(-\frac{Q(x-y)}{\lambda}\right)
	=
	\begin{cases}
		q-1, & \text{if } Q(x-y)=0,\\
		-1, & \text{if } Q(x-y)\ne0.
	\end{cases}
	\]
	The number of pairs \((x,y)\in E^2\) with \(Q(x-y)=0\) is $|E|(1+a_0),$
	where the term \(1\) accounts for the diagonal. Hence
	\begin{align*}
		0
		&\le
		q^{2m-1}|E|
		-
		q^{m-1}
		\left(
		(q-1)|E|(1+a_0)
		-
		\bigl(|E|^2-|E|(1+a_0)\bigr)
		\right) \\
		&=
		q^{2m-1}|E|
		-
		q^{m-1}
		\left(
		q|E|(1+a_0)-|E|^2
		\right).
	\end{align*}
	Dividing by \(q^{m-1}|E|\), we obtain $0\le q^m-q(1+a_0)+|E|.$
	Therefore
	\[
	a_0\le q^{m-1}-1+\frac{|E|}{q}
	\le q^{m-1}+\frac{|E|}{q}.
	\]
	This proves the lemma.
\end{proof}

We now prove the main theorem. The argument is a Delsarte linear programming
bound. The variables \(a_0\) and \(a_t\), \(t\in T\), are the inner
distribution variables, while Proposition~\ref{prop-lp-constraints} gives
the eigenvalue constraints. Lemma~\ref{lem-7} supplies the dual feasible
solutions needed to bound the objective.

\begin{proof}[Proof of Theorem~\ref{mainthm}]
	Let $T=\Delta_Q(E)\cap\F_q^\times.$
	Since \(0=Q(x-x)\in\Delta_Q(E)\), we have $	|\Delta_Q(E)|=1+|T|.$
	It is therefore enough to prove the following contrapositive statement: if $|T|\le \alpha(q-1),$
	then $|E|\le C'_\alpha q^{m+\frac13}$
	for some constant \(C'_\alpha\) depending only on \(\alpha\).
	
	Assume from now on that $|T|\le \alpha(q-1).$
	The inner distribution of \(E\) gives a feasible point of the following
	linear program:
	\[
	\begin{array}{ll}
		\text{maximize}
		&
		1+a_0+\displaystyle\sum_{t\in T}a_t,\\[2mm]
		\text{subject to}
		&
		a_0\ge0,\\[1mm]
		&
		a_t\ge0,\qquad t\in T,\\[1mm]
		&
		q^m+\varepsilon_Q
		\left(
		-1-a_0+\displaystyle\sum_{t\in T}K(st)a_t
		\right)
		\ge0,
		\qquad s\in\F_q^\times.
	\end{array}
	\tag{LP}
	\]
	The objective value is precisely $1+a_0+\sum_{t\in T}a_t=|E|.$
	When \(\varepsilon_Q=-1\), Lemma~\ref{lem-9} gives the additional constraint
	\[
	a_0\le q^{m-1}+\frac{|E|}{q}.
	\tag{Z}
	\]
	
	Let \(c_\alpha>0\) be the constant from Lemma~\ref{lem-7}. By enlarging
	\(q_\alpha\) if necessary, we may assume throughout that $c_\alpha q^{-\frac{1}{3}}\le1$
	and $c_\alpha q^{-\frac{1}{3}}-\frac{2}{q}
	\ge
	\frac{c_\alpha}{2}q^{-\frac{1}{3}}.$
	
	We distinguish two cases according to the sign of \(\varepsilon_Q\).
	
	\medskip
	
	\noindent
	\textbf{Case 1: \(\varepsilon_Q=1\).}
	
	In this case, the Delsarte constraints become
	\[
	q^m-1-a_0+\sum_{t\in T}K(st)a_t\ge0
	\]
	for every \(s\in\F_q^\times\).
	
	If \(T=\emptyset\), then the above inequality gives $a_0\le q^m-1.$
	Thus $|E|=1+a_0\le q^m,$
	which is more than sufficient.
	
	Now suppose that \(T\ne\emptyset\). Since $1\le |T|\le \alpha(q-1),$
	Lemma~\ref{lem-7} gives a probability vector \(\mu_-\) on
	\(\F_q^\times\) such that $\sum_{s\in\F_q^\times}\mu_-(s)K(st)
	\le
	-c_\alpha q^{-\frac{1}{3}}$
	for every \(t\in T\). This probability vector is the dual feasible
	certificate for the linear program in the present case.
	
	Averaging the Delsarte constraints against \(\mu_-\), we obtain
	\begin{align*}
		0
		&\le
		q^m-1-a_0
		+
		\sum_{t\in T}
		\left(
		\sum_{s\in\F_q^\times}\mu_-(s)K(st)
		\right)
		a_t \\
		&\le
		q^m-1-a_0
		-
		c_\alpha q^{-\frac{1}{3}}
		\sum_{t\in T}a_t.
	\end{align*}
	Therefore
	\[
	a_0+c_\alpha q^{-\frac{1}{3}}\sum_{t\in T}a_t
	\le
	q^m-1.
	\]
	Since \(c_\alpha q^{-\frac{1}{3}}\le1\), we have
	\[
	c_\alpha q^{-\frac{1}{3}}
	\left(
	a_0+\sum_{t\in T}a_t
	\right)
	\le
	a_0+c_\alpha q^{-\frac{1}{3}}\sum_{t\in T}a_t.
	\]
	Using \(|E|=1+a_0+\sum_{t\in T}a_t\), it follows that
	\[
	c_\alpha q^{-\frac{1}{3}}(|E|-1)
	\le
	q^m-1.
	\]
	Hence $|E|\le 1+c_\alpha^{-1}q^{m+\frac13}.$
	Thus $|E|\ll_\alpha q^{m+\frac13}.$
	
	\medskip
	
	\noindent
	\textbf{Case 2: \(\varepsilon_Q=-1\).}
	
	In this case, the Delsarte constraints become
	\[
	q^m+1+a_0-\sum_{t\in T}K(st)a_t\ge0
	\]
	for every \(s\in\F_q^\times\).
	
	If \(T=\emptyset\), then \(|E|=1+a_0\). By Lemma~\ref{lem-9}, $a_0\le q^{m-1}+\frac{|E|}{q}.$
	Therefore $|E|
	\le
	1+q^{m-1}+\frac{|E|}{q},$
	and hence $|E|
	\le
	\frac{q^m+q}{q-1}.$
	This is again more than sufficient.
	
	Now suppose that \(T\ne\emptyset\). By Lemma~\ref{lem-7}, there exists a
	probability vector \(\mu_+\) on \(\F_q^\times\) such that $\sum_{s\in\F_q^\times}\mu_+(s)K(st)
	\ge
	c_\alpha q^{-\frac{1}{3}}$
	for every \(t\in T\). This gives the dual feasible certificate in the
	negative-sign case.
	
	Averaging the Delsarte constraints against \(\mu_+\), we get
	\begin{align*}
		0
		&\le
		q^m+1+a_0
		-
		\sum_{t\in T}
		\left(
		\sum_{s\in\F_q^\times}\mu_+(s)K(st)
		\right)
		a_t \\
		&\le
		q^m+1+a_0
		-
		c_\alpha q^{-\frac{1}{3}}
		\sum_{t\in T}a_t.
	\end{align*}
	Thus
	\[
	c_\alpha q^{-\frac{1}{3}}\sum_{t\in T}a_t
	\le
	q^m+1+a_0.
	\]
	Since \(c_\alpha q^{-\frac{1}{3}}\le1\), we have
	\begin{align*}
		c_\alpha q^{-\frac{1}{3}}(|E|-1)
		&=
		c_\alpha q^{-\frac{1}{3}}
		\left(
		a_0+\sum_{t\in T}a_t
		\right) \\
		&\le
		q^m+1+(1+c_\alpha q^{-\frac{1}{3}})a_0 \\
		&\le
		q^m+1+2a_0.
	\end{align*}
	Using Lemma~\ref{lem-9}, we obtain
	\[
	c_\alpha q^{-\frac{1}{3}}(|E|-1)
	\le
	q^m+1+2q^{m-1}+\frac{2|E|}{q}.
	\]
	Rearranging gives
	\[
	\left(
	c_\alpha q^{-\frac{1}{3}}-\frac{2}{q}
	\right)|E|
	\le
	q^m+1+2q^{m-1}+c_\alpha q^{-\frac{1}{3}}.
	\]
	By our choice of \(q_\alpha\),
	\[
	c_\alpha q^{-\frac{1}{3}}-\frac{2}{q}
	\ge
	\frac{c_\alpha}{2}q^{-\frac{1}{3}}.
	\]
	Therefore
	\[
	|E|
	\le
	\frac{2}{c_\alpha}
	q^{\frac{1}{3}}
	\left(
	q^m+1+2q^{m-1}+c_\alpha q^{-\frac{1}{3}}
	\right)
	\ll_\alpha
	q^{m+\frac13}.
	\]
	
	Combining the two cases, we have shown that if $|\Delta_Q(E)\cap\F_q^\times|\le \alpha(q-1),$
	then $|E|\le C'_\alpha q^{m+\frac13}.$
	Choosing the constant \(C_\alpha\) in the statement of the theorem larger
	than \(C'_\alpha\), we obtain the contrapositive: if $|E|\ge C_\alpha q^{m+\frac13},$
	then
	\[
	|\Delta_Q(E)\cap\F_q^\times|>\alpha(q-1).
	\]
	Since \(0\in\Delta_Q(E)\), this implies
	\[
	|\Delta_Q(E)|>1+\alpha(q-1).
	\]
	The proof of Theorem~\ref{mainthm} is complete.
\end{proof}

\section*{Acknowledgements}
Tao Zhang is partially supported by National Natural Science Foundation of China (Grant No. 12571357),
Natural Science Basic Research Program of Shaanxi  (Program No. 2025JC-YBMS-048).

\end{document}